\documentclass[a4paper,11pt]{amsart}
\usepackage{amssymb,amsmath}
\usepackage[margin=3.5cm]{geometry}
\usepackage{graphicx}
\usepackage{hyperref}
\usepackage[utf8]{inputenc}
\usepackage{cmtt}

\usepackage{todonotes}

\newtheorem{theorem}{Theorem}

\newtheorem{corollary}[theorem]{Corollary}

\newtheorem{proposition}[theorem]{Proposition}

\theoremstyle{remark}

\let\leq\leqslant
\let\geq\geqslant

\let\setminus\smallsetminus
\let\epsilon\varepsilon

\newcommand{\floor}[1]{{\left\lfloor #1 \right\rfloor}}

\newcommand{\Nat}{\mathbb{N}}

\makeatletter
\let\old@setaddresses\@setaddresses
\def\@setaddresses{\bgroup\parindent 0pt\let\scshape\relax\old@setaddresses\egroup}
\makeatother

\DeclareMathOperator{\ch}{ch}
\DeclareMathOperator{\olch}{pnt}

\begin{document}

\title{Chip games and paintability}

\author[L.~Duraj]{Lech Duraj}
\author[G.~Gutowski]{Grzegorz Gutowski}
\author[J.~Kozik]{Jakub Kozik}

\address{Theoretical Computer Science Department, Faculty of Mathematics and Computer Science, Jagiellonian University, Krak\'{o}w, Poland}
\email{\mtt\{duraj,gutowski,jkozik\mtt\}@tcs.uj.edu.pl}

\thanks{This research is supported by: Polish National Science Center UMO-2011/03/D/ST6/01370.}

\begin{abstract}
We prove that the difference between the paint number and the choice number of a complete bipartite graph $K_{N,N}$ is $\Theta(\log \log N)$.
That answers the question of Zhu (2009) whether this difference, for all graphs, can be bounded by a common constant.
By a classical correspondence, our result translates to the framework of on-line coloring of uniform hypergraphs. 
This way we obtain that for every on-line two coloring algorithm there exists a $k$-uniform hypergraph with $\Theta(2^k)$ edges on which the strategy fails.
The results are derived through an analysis of a natural family of chip games.
\end{abstract}

\maketitle

\section{Introduction}

We begin with a presentation of connections between paintability and on-line coloring of regular hypergraphs.
In Section~\ref{sec:chips} we show how to express a common problem in these areas using a natural family of simple two person chip games.
Main results, presented in terms of these chip games, follow in Section~\ref{sec:main}.
Throughout the paper $\log$ is the logarithm function to base $2$.

\subsection{Paintability}

The classical notion of choosability (list colorability) has been brought to the on-line setting by means of
  the on-line list coloring game introduced by Schauz~\cite{Schauz09} and Zhu~\cite{Zhu09}.
Given a finite graph $G=(V,E)$ and a function $f:V \to \Nat$, two players, Lister and Painter, play the \emph{on-line list coloring game} in the following way.
In the $i$-th round Lister presents a nonempty set of vertices $V_i \subset V \setminus \bigcup_{j=1}^{i-1} X_j$ and Painter chooses $X_i$ that is both a subset of $V_i$ and an independent set in $G$.
After $i$ rounds, vertices in $\bigcup_{j=1}^{i} X_j$ are \emph{colored}.
If a vertex $v$ belongs to exactly $l$ of the sets $V_1, \ldots, V_i$ we say that $v$ has $l$ \emph{permissible colors} after $i$ rounds.
Lister wins the game if after some round there exists an uncolored vertex $v$ with $f(v)$ permissible colors.
If this does not happen, then eventually all vertices are colored and Painter wins the game.
We say that $G$ is \emph{$f$-paintable} when Painter has a winning strategy.
Graph $G$ is \emph{$k$-paintable} if it is $f$-paintable for a constant function $f(x) = k$.
The smallest such $k$ is called the \emph{paint number} of $G$ and is denoted by $\olch(G)$.

Surprisingly many results proved for choosability turn out to be valid for paintability as well.
In particular, the paint number of any planar graph is at most 5~\cite{Schauz09}. 
Kernel method of Galvin generalizes in a straightforward way, hence, the paint number of the line graph of any bipartite graph $G$ equals the maximum degree in $G$.
By the result of Schauz~\cite{Schauz10}, proofs using Combinatorial Nullstellensatz are also valid in the on-line setting.
For example, if a graph $G$ admits an orientation in which the number of even Eulerian subgraphs differs from the number of odd Eulerian subgraphs and $f(v) \geq d^+(v)+1$, then G is $f$-paintable.

One of the most intriguing questions concerning paintability is the relation of the paint number to the choice number.
Clearly, for any graph $G$ we have $\ch(G) \leq \olch(G)$.
There are examples of graphs with $\ch(G) < \olch(G)$, but no example with $\ch(G)+1 < \olch(G)$ was presented so far.
The main result of this paper provides the first family of examples with an arbitrarily large difference between the paint number and the choice number.
We achieve this by providing the following tight estimate of the paint number of a complete bipartite graph $K_{N,N}$.

\begin{theorem}
	\[
    \olch(K_{N,N}) = \log N + O(1)\textrm{.}
	\]
\end{theorem}

It is commonly known that $\ch(K_{N,N}) = (1+o(1)) \log N$.
To observe a difference between $\ch(K_{N,N})$ and $\olch(K_{N,N})$ we need a more precise estimate.
Erd\H{o}s, Rubin and Taylor~\cite{ERT80} observed that $\ch(K_{N,N})$ is closely related to another well studied parameter.
Let $m(k)$ be the minimum number of edges in a $k$-uniform hypergraph that is not $2$-colorable.
Easy encoding (see~\cite{ERT80}) shows that the minimum $N$, for which $\ch(K_{N,N}) \geq k$, satisfies $m(k) \leq N \leq 2 \cdot m(k)$.
The best known lower bound for $m(k)$, proved by Radhakrishnan and Srinivasan~\cite{RS00}, is $m(k) = \Omega (\sqrt{\frac{k}{\log k}} \cdot 2^{k} )$.
As a consequence, we get $\ch(K_{N,N}) \leq \log N - (\frac{1}{2} +o(1))\log \log N$ and the following corollary.

\begin{corollary}
\[
  \olch(K_{N,N}) - \ch(K_{N,N}) = \Omega(\log \log N)\textrm{.}
\]
\end{corollary}

This difference can be arbitrarily large, but still it is logarithmic as a function of the choice number.
By the result of Alon~\cite{Alon93}, we know that for any graph, the paint number is bounded from above with an exponential function of the choice number.
It is an interesting question to decide if the paint number can be bounded from above with a polynomial function of the choice number.

\subsection{On-line coloring of uniform hypergraphs}

The aforementioned parameter $m(k)$ has been introduced by Erd\H{o}s and Hajnal~\cite{ErdHaj61}.
Currently, the best bounds for $m(k)$ are
\[
	c \sqrt{\frac{k}{\log k}} \cdot 2^{k} \leq m(k) \leq (1+o(1)) \frac{e \ln 2}{4}k^2 \cdot 2^k.
\]
The lower bound has been proved by Radhakrishnan and Srinivasan~\cite{RS00}, the upper bound by Erd\H{o}s~\cite{Erd1964}.

An interesting version of the problem of hypergraph coloring is the \emph{on-line coloring}.
Within this setting vertices are presented on-line, one by one, each vertex $v$ together with indices of all edges that contain $v$.
A color for $v$ needs to be assigned immediately and can not be changed later.
The goal is to avoid monochromatic edges.

The problem of \emph{on-line $2$-coloring of $k$-uniform hypergraphs} can be formalized as a game between Presenter and Colorer.
The game is parametrized by two numbers: the cardinality of edges $k$, and the number of edges $N$.
Values of these parameters are known to both players before the game.
In each round, Presenter reveals one vertex and declares in which edges it is contained.
Presenter can not add vertices to edges which already contain $k$ vertices.
Colorer must immediately assign color $0$ or $1$ to the presented vertex.
Presenter wins when there exists a monochromatic edge containing $k$ vertices.
Colorer wins when all vertices have been revealed (i.e.\ all $N$ edges contain $k$ vertices each) and no edge is monochromatic.
The number $m^{OL}(k)$ is the smallest $N$ for which Presenter has a winning strategy in the on-line $2$-coloring game on $k$-uniform hypergraphs with $N$ edges.
The best bounds for $m^{OL}(k)$ shown before were
\[
	2^{k-1} \leq m^{OL}(k) \leq m(k) \leq (1+o(1)) \frac{e \ln 2}{4}k^2 \cdot 2^k.
\]
The lower bound has been obtained independently by several authors and can be considered as a derandomization of the analogous bound for $m(k)$ by Erd\H{o}s~\cite{Erd1963}.
An obvious upper bound is obtained by a strategy that presents (in any order) a small $k$-uniform hypergraph that is not $2$-colorable.

Aslam and Dhagat~\cite{AsDha93} gave an upper bound $m^{OL}(k) < k \cdot \phi^{2k}$, where $\phi= \frac{1+\sqrt{5}}{2}$ is the golden ratio.
Although their bound was weaker than already known upper bound for $m(k)$, their strategy has the advantage of being explicit.
Note that it is hard to efficiently construct such hypergraphs.
The best known deterministic construction by Gebauer~\cite{Gebauer13} gives such $k$-uniform hypergraphs with $2^{k+O(k^{2/3})}$ edges.

The classical correspondence between $2$-coloring of uniform hypergraphs and choosability of complete bipartite graphs translates also into the on-line setting.
Therefore our result on the paint number of a graph $K_{N,N}$ implies the following corollary.

\begin{corollary}
\[
  m^{OL}(k) = \Theta(2^k)\textrm{.}
\]
\end{corollary}

Moreover, we describe an explicit, deterministic strategy for Presenter.

\section{Chip games}\label{sec:chips}

Aslam and Dhagat~\cite{AsDha93} modeled on-line coloring of hypergraphs by a specific chip game.
We use the following, slightly modified variant which corresponds to the on-line list coloring game on $K_{N,N}$.

\subsection{General chip game}

A \emph{$(k,N)$ chip game} is a finite two player game with perfect information played on a board consisting of two directed paths on $k+1$ vertices each.
Consecutive vertices on each path are indexed with numbers $k, k-1, \ldots, 0$.
During the game, every vertex on each path is occupied by a nonnegative number of chips.
In the starting configuration only vertices with index $k$ are occupied, each of them by exactly $N$ chips.
In each round Pusher chooses an arbitrary nonempty set of chips and moves them all one step forward (towards $0$).
Then, Remover chooses one of the paths and removes from the board all the chips that have been moved in this round on the chosen path.
Thus, the number of chips decreases in each round.
The game ends when either there is a chip on the last vertex on any path (in this case Pusher wins) or all chips are removed from the board (in this case Remover wins).

\subsection{Chip game -- paintability}

An on-line list coloring game on $G=K_{N,N}$ with lists of length $k$ can be rephrased in terms of a $(k,N)$ chip game.
Paths in the chip game correspond to parts of $G$, chips correspond to uncolored vertices of $G$
and a position of a chip on its path is $k$ minus the number of its permissible colors, or equivalently the number of colors that are yet to come for that vertex.
At the beginning, no vertex of the graph has any permissible colors and indeed all chips are at position $k$.
In each round, Lister presents some set $V_j$ of uncolored vertices and Painter chooses an independent subset $X_j$ of $V_j$.
Clearly, $X_j$ contains vertices in only one part of $K_{N,N}$. 
On the other hand, we may assume that $X_j$ contains all vertices moved in that part.
It is easy to check that any winning strategy can be modified to always remove all chips moved in one of the parts.
As a result, one part of $V_j$ gets colored, while the number of colors that are yet to come for any vertex in the other part of $V_j$ decreases by one.
Presenting the set $V_j$ by Lister is modeled in the chip game as Pusher moving chips corresponding to vertices in $V_j$.
Coloring one part by Painter is modeled as Remover choosing the path from which vertices are removed.
Chips that correspond to vertices in $X_j$ are removed from the board and those that correspond to $V_j \setminus X_j$ are moved one step forward.
Thus, after each move, the position of a chip matches the number of permissible colors that are yet to come for a corresponding vertex.

By the above discussion, instead of asking what is the paint number of $K_{N,N}$ we can equivalently ask what is the smallest number $k$ for which Remover wins $(k,N)$ chip game.
Similarly, instead of asking what is the minimum $N$ for which $K_{N,N}$ is not $k$-paintable we can ask what is the minimum $N$ for which Pusher wins $(k,N)$ chip game.

\subsection{Chip game - on-line coloring of uniform hypergraphs}

Suppose that Remover has a winning strategy in a $(k,N)$ chip game.
Colorer can use this strategy to win an on-line $2$-coloring game on $k$-uniform hypergraphs with $N$ edges.
Colorer represents each edge of the hypergraph by two chips, one on each path.
When Presenter reveals a new vertex $v$ and declares to which edges $v$ belongs, Colorer makes a Pusher's move in the chip game.
Colorer moves all chips representing edges that contain $v$.
If the winning strategy for Remover removes chips from the first path, then Colorer colors $v$ with $0$.
Otherwise, Colorer uses color $1$.
During the game, a monochromatic edge containing $i$ ($0<i<k$) vertices is represented by a single chip on position $k-i$ in the chip game.
Eventually, Remover wins the game which means that no edge with $k$ vertices is monochromatic.

Similarly, Pusher's winning strategy in a $(k,N)$ chip game can be used by Presenter to win an on-line $2$-coloring game on $k$-uniform hypergraphs with $2N$ edges.
Presenter represents each chip in the chip game by an edge of the hypergraph.
When winning strategy for Pusher moves chips in a set $V$, Presenter reveals a vertex $v$ belonging to all edges representing chips in $V$.
If Colorer colors $v$ with $0$, then Presenter removes, as Remover, in the chip game chips on the first path.
Otherwise, Presenter removes chips moved on the second path.
A chip on position $k-i$ is represented by a monochromatic edge containing $i$ vertices.
Hence, Pusher's win means that eventually there is a monochromatic edge containing $k$ vertices.

\subsection{$1$-restricted game}

Aslam and Dhagat~\cite{AsDha93} considered also a \emph{$c$-restricted} variant of the game.
In this variant, Pusher is allowed to move at most $c$ chips on each path in a single round.
In the context of hypergraph coloring, this restriction corresponds to a bound on the vertex degree of the presented hypergraph, i.e.\ hypergraphs presented in a $c$-restricted chip game have maximum vertex degree at most $2c$.
Clearly, for a fixed $c$ and a large enough $k$, every $k$-uniform hypergraph of maximum vertex degree at most $2c$ is $2$-colorable.
Similar statement does not hold in the on-line setting.
Aslam and Dhagat~\cite{AsDha93} showed that for $N \geq (3+2 \sqrt{2})^k$ Pusher has a winning strategy in a $1$-restricted $(k,N)$ chip game.
Let $t_c(k)$ be the minimum value of $N$ for which Pusher has a winning strategy in a $c$-restricted $(k,N)$ chip game.
We determine the asymptotics of the threshold function $t_1(k)$ up to a multiplicative factor.

\begin{theorem}
\label{thm:1-restricted}
\[
  t_1(k) = \Theta(F_{2k})\textrm{,}
\]
where $F_{2k}$ is the $(2k)$-th Fibonacci number.
\end{theorem}
It is an interesting open problem, stated already by Aslam and Dhagat~\cite{AsDha93}, to determine $\lim_{k\to \infty} t_c(k)^\frac{1}{k}$ for $c \geq 2$.

\subsection{Maker-Maker-Breaker game}
Pegden~\cite{PegdenPC} suggested an interesting variant of Maker-Breaker games.
In order to balance a Maker-Braker game in which Breaker wins, it is usual to allow Maker to make more subsequent moves for one Breaker move.
Pegden suggested an alternative solution, that is, to replace a single Maker by a coalition of Makers.
In a simple setting we have two Makers (red and blue), and one Breaker (black).
They play on a family of $N$ disjoint $k$-sets of vertices.
In one turn, in a prescribed order, each player picks some uncolored vertex and paints it with his color.
Makers cooperate to achieve the common goal which is to color one of the $k$-sets monochromatic red or blue, while Breaker wants to prevent it.
Natural question is, how large $N$ should be in order for Makers to have a winning strategy.
The game is naturally modeled by a 1-restricted $(k,N)$ chip game with a modification that allows Remover to delete not only one of the moved chips, but any chip from the board.
As it turns out, the answer in this modified setting is analogous to the result of Theorem~\ref{thm:1-restricted} -- the threshold function is also $\Theta(F_{2k})$.
The lower bound, which corresponds to a strategy for Breaker is valid in the modified setting.
The proof of the upper bound from Proposition~\ref{prop:1-restricted-upper} does not seem to generalize.
An alternative, more technical, proof of this bound which is also valid in the modified setting will appear in the full version of the paper.

\section{Main results}\label{sec:main}

\subsection{General game.}

We assign weight $2^{-i}$ to a chip on vertex $i$ on any path (recall that vertices on each path are indexed from $k$ to $0$).
Weight of a group of chips is the sum of weights of chips in this group.
For a warm-up we present a simple strategy which allows Pusher to win $(k, k \cdot 2^{k-1})$ chip game.
It is weaker than our subsequent result but better than anything that has been known so far.

\begin{proposition}
Pusher has a winning strategy in a $(k,k \cdot 2^{k-1})$ chip game.
\end{proposition}

\begin{proof}
We define a \emph{brick} to be a group of exactly $2^{i-1}$ chips occupying vertex $i$ on any path.
Observe that any such brick has weight $\frac{1}{2}$.
At the beginning of a $(k,k \cdot 2^{k-1})$ chip game there are $k$ disjoint bricks on the starting vertex on each path.
The strategy for Pusher is as follows.
In each round, he selects a brick with lowest possible position on each path, and moves those two bricks forward.
Remover responds by removing one of those bricks.
The other brick moves one step forward, say from position $i$ to $i-1$.
Assuming that $i-1>0$ (otherwise the game is won by Pusher), those $2^{i-1}$ chips on position $i-1$ can be split into two disjoint bricks.
As a result, one of the bricks gets removed, but the other one moves forward and splits into two.
The number of disjoint bricks on the board stays the same.
Consider the distribution of bricks on one path and assume that the lowest occupied position on that path is $i$.
We claim that if $i < k$ then there are at most two disjoint bricks on position $i$ and at most one brick on each position from $k-1$ to $i+1$.
This observation is clearly true in the starting position.
In each subsequent round, a brick with the lowest possible position is removed and possibly two bricks appear on a previously unoccupied vertex.
Hence, the property is preserved during the whole game.

Suppose that Pusher can not make a move according to the described strategy.
It means that there are no more bricks on one of the paths.
Since the number of bricks is constant, all $2k$ bricks are on the other path.
By the observation above, that path contains at most $k$ bricks on positions from $k-1$ to $1$.
Moreover, at least one brick must have left the first vertex and hence there are at most $k-1$ bricks on position $k$.
Altogether, that path contains at most $2k-1$ bricks which gives a contradiction.
\end{proof}

\begin{proposition}
Pusher has a winning strategy in a $(k, 8 \cdot 2^{k})$ chip game. 
\end{proposition}

\begin{proof}
For the specified parameters we describe a winning strategy for Pusher.
We are going to use the same weights as in the previous proof: a chip on vertex $i$ has weight $2^{-i}$.
Pusher is going to play his strategy in phases.
Before and after each phase the following invariants hold.
With each invariant we give a short idea how the strategy is going to preserve it.
These ideas will be described in full detail later.
\smallskip

\noindent \textbf{Weight Invariant:} \emph{The total weight of chips on any path equals $8$.}
The strategy will preserve this invariant using the following doubling technique.
When the total weight of a single path differs from $8$ by $w$, the strategy will move chips of total weight $w$.
This way the difference from $8$ either doubles or settles to zero.
\smallskip

\noindent \textbf{Consecutive Vertices Invariant:} \emph{Chips on any path occupy at most two consecutive vertices on the path: $h$ and $h-1$.}
Thus, the state of a single path can be described by two values: the highest occupied vertex $h$ ($0 < h \leq k$) and the total weight of chips on that vertex $W$ ($0 < W \leq 8$).
The \emph{distance} of a path $D$ is defined by $D = 8h + W$.
The strategy will preserve this invariant by moving the chips from the highest occupied vertex first.
\smallskip

\noindent \textbf{Difference Invariant:} \emph{Let $D_1$, $D_2$ be the distances of the first and the second path. The difference $|D_1 - D_2|$ is at most $8$.}
The distances of both paths will decrease in each phase.
The strategy will preserve this invariant by guaranteeing that in each phase the distances of both paths do not decrease by more than $8$ and that the distance of a path with the higher distance decreases more than the distance of the other path.
\smallskip

Without loss of generality, assume that the distance of the first path is greater or equal to the distance of the second path and that $8h_1 + W_1 = D_1 \geq D_2 = 8h_2 + W_2$.
As the difference between $D_1$ and $D_2$ is at most $8$ and $W_1,W_2 \leq 8$, we get $h_2 \leq h_1 \leq h_2 + 1$.
As the total weight of chips on each path equals $8$, the number of chips on highest occupied vertex on any path is even.
Let $\omega = 2^{-h_2}$ be the weight of a chip occupying vertex $h_2$ on the second path.
During the phase Pusher chooses some groups of chips to be moved.
Saying that Pusher \emph{draws} chips of some weight $\alpha$ from a path we mean that he selects a group of chips of total weight $\alpha$ adding chips one by one from the highest possible position.
In the first round of any phase Pusher draws chips of total weight $8 \omega$ from the first path, and draws chips of total weight $2 \omega$ from the second path.
The rest of the phase depends on the Remover's answer in the first round.
\smallskip

\textbf{Case 1.} Remover responds to the first move by removing chips from the first path.
After that, the total weight of the first path is $8-8\omega$.
Pusher selects one of the two chips moved on the second path to be the \emph{running chip}.
Other chips on the second path have total weight $8$ and none of them will be moved anymore in this phase.
In the $i$-th round in this phase (for $i=2,3,\ldots$) Pusher draws chips of total weight $4\cdot2^{i-1}\omega$ from the first path and moves them together with the running chip.
The phase ends when Remover removes the running chip.
Assume that it happens in the $m$-th round and observe that $m \leq h_2$, for otherwise the running chip gets to vertex $0$ and Pusher wins.
The total weight that is drawn from the first path is
$$
8\omega + 4\cdot\sum_{i=2}^{m}2^{i-1}\omega = 4\cdot2^m\omega \leq 4\cdot2^{h_2}\cdot2^{-h_2} = 4\textrm{,}
$$
which guarantees that Pusher does not run out of the chips.
When the phase ends, the total weight of any path equals $8$, as $8 - 8\omega - 4\cdot\sum_{i=2}^{m}2^{i-1}\omega + 4\cdot2^m\omega = 8$.
The drawing technique guarantees that chips on any path occupy at most two consecutive vertices.
The distance of the second path has decreased by exactly $2\omega$.
Let $B=4\cdot2^m\omega \leq 4$ be the total weight of chips drawn from the first path during this phase.
Case analysis shows that the distance of the first path decreases by:
\[
 \begin{cases}
		B, & \text{for } B \leq W_1; \\
    2B-W_1, & \text{for } \frac{B}{2} \leq W_1 \leq B;\\
    B+W_1, & \text{for } W_1 \leq \frac{B}{2}.\\
	\end{cases}
\]
In particular, it decreases by at least $4 \omega$ and by at most $6$.
As a result, the difference between the distances of both paths is still at most $8$.
\smallskip

\textbf{Case 2.} Remover responds to the first move by removing chips from the second path.
After that, Pusher selects running chips: one or two chips of total weight $2\omega$ from those that were moved in the first round.
Then, the total weight of the rest of the chips on the first path is $8+6\omega$.
The total weight of the second path is $8-2\omega$.
In the $i$-th round in this phase (for $i=2,3,\ldots$) Pusher draws chips of total weight $3\cdot2^{i-1} \omega$ from the first path and chips of total weight $2^{i-1}\omega$ from the second path and moves them together with the running chips.
The phase ends when Remover removes the running chips.
Assume that it happens in the $m$-th round and observe that $m \leq h_1 \leq h_2 + 1$.
The total weight that is drawn from the first path is
$$
8\omega + 3\cdot\sum_{i=2}^{m}2^{i-1}\omega = 2\omega + 3\cdot2^m\omega \leq 2\omega + 3\cdot2^{h_2+1}\cdot2^{-h_2} \leq 7\textrm{,}
$$
while the total weight that is drawn from the second path is
$$
2\omega + \sum_{i=2}^{m}2^{i-1}\omega = 2^m\omega \leq 2^{h_2+1}\cdot2^{-h_2} = 2\textrm{,}
$$
which guarantees that Pusher does not run out of the chips.
When the phase ends, the total weight of each path equals $8$ and chips on each path occupy at most two consecutive vertices.
By the same argument as in Case 1, the distance of the second path decreases by at most $3\cdot2^{m-1}\omega$.
For $A=3 \cdot 2^m \omega$, case analysis shows that the distance of the first path decreases by: 
\[
	\begin{cases}
		A+2\omega, & \text{for } A+2\omega \leq W_1; \\
    W_1, & \text{for } \frac{A}{2}+2\omega \leq W_1 \leq A+2\omega;\\
    A+4\omega-W_1, & \text{for } W_1 \leq \frac{A}{2}+ 2\omega.
\\
	\end{cases}
\]
In particular it decreases by at least $2\omega+3\cdot 2^{m-1} \omega$ and by at most $2\omega+3\cdot 2^{m} \omega \leq 8$.
Thus, the difference between the distances of both paths is still at most $8$.
\smallskip

The invariants guarantee that the strategy can be played by Pusher as long as there is no chip on vertex $0$ on any path.
Hence, it is a winning strategy for Pusher.
\end{proof}

\subsection{$1$-restricted game}

Recall the definition of the Fibonacci sequence $(F_j)_{j\in\Nat}$:
$$\begin{aligned}
  F_0 &= 0,\\
  F_1 &= 1,\\
  F_{j+2} &= F_{j+1} + F_{j}.
\end{aligned}$$

\begin{proposition}
Remover has a winning strategy in a $1$-restricted $(k,\floor{\frac{F_{2k+1}-1}{2}})$ chip game.
\end{proposition}

\begin{proof}
We present a winning strategy for Remover that works under the assumption that in each round Pusher moves exactly one chip on each path.
Remover responds by removing the one of the two moved chips that is on a lower position.
In case of a draw he removes the one from the first path.

We assign weight $F_{2(k-i)+2}$ to a chip on the $i$-th vertex of the first path, and weight $F_{2(k-i)+1}$ to a chip on the $i$-th vertex of the second path.
In the starting position, each chip is of weight $1$ and the total weight of all chips is smaller than $F_{2k+1}$.

Let us examine the change of the total weight in one round.
Assume that a chip is moved one step forward so that its weight changes from $F_j$ to $F_{j+2}$.
Then, the removed chip had weight at least $F_{j+1}$.
Altogether, the total weight of all chips on one path increases by $F_{j+2}-F_j$ and decreases by at least $F_{j+1}$ on the second path.
Hence, by the Fibonacci recurrence, the total weight of all chips does not increase during the game and it is always smaller than $F_{2k+1}$.
As a result, no chip can get to vertex $0$ since its weight would be at least $F_{2k+1}$.
\end{proof}

\begin{proposition}
\label{prop:1-restricted-upper}
Pusher has a winning strategy in a $1$-restricted $(k,F_{2k}+2k F_{k+1})$ chip game.
\end{proposition}

\newcommand{\TW}{\mathcal{W}}
\begin{proof}
We split the chips into two types: some of them will be organized in groups called \emph{towers}, while the rest will be kept in two \emph{buckets}.
Pusher will move only the tower chips, and the bucket ones will be transferred to towers when necessary.
Each path has its own bucket.
An \emph{$(a,b)$-tower} (with $a, b \geq 0$) consists of $a$ chips on the first path on vertices $k, k-1, \ldots, k-(a-1)$ and $b$ chips on the second path on vertices $k, k-1, \ldots,k-(b-1)$.
The \emph{size} of an $(a,b)$-tower is $a+b$.
We remark that towers of size $0$ or $1$ (and, in general, any $(a,0)$- or $(0,b)$-towers) will appear only briefly.
Observe that if at any moment a $(k,a)$- or $(b,k)$-tower appears, Pusher wins the game.
In particular, when a tower of size $2k-1$ or bigger appears, Pusher wins the game.

We start the game with $2k-2$ $(1,1)$-towers.
The Pusher's strategy uses three types of actions.
\smallskip

\noindent \textbf{Rebuild:}
If there is an $(a,0)$- or $(0,b)$-tower, Pusher takes a chip from the corresponding bucket and transfers it to the tower, changing the tower to $(a,1)$- or $(1,b)$-tower.
\smallskip

\noindent \textbf{Exchange:}
If no rebuild is possible and if there are two towers of equal size $s$, an $(a,b)$-tower and an $(a',b')$-tower with $s = a+b = a'+b'$ and $a \neq a'$, then Pusher reorganizes them into an $(a,b')$- and $(a',b)$-tower.
Observe that this action replaces two towers of size $s$ each with a tower of size $s+i$ and another one of size $s-i$ for some positive $i$.
\smallskip

\noindent \textbf{Advance:}
If no rebuild and no exchange is possible and if there are two $(a,b)$-towers for some $a,b > 0$, then Pusher makes an actual move.
He selects two $(a,b)$-towers with maximum possible size $s=a+b$.
Then, he moves forward two top-most chips from the first tower: a chip from vertex $k-(a-1)$ on the first path and a chip from vertex $k-(b-1)$ on the second path.
One of the chips is removed by Remover.
Pusher transfers the second chip to the second tower.
As a result, two towers of size $s$ are replaced by a tower of size $s-2$ and another one of size $s+1$.
\smallskip

We prove that as long as the game continues, one of these actions is always possible.
Suppose to the contrary that no action is possible.
Since rebuild is impossible, every tower has positive number of chips on both paths.
Since Pusher didn't win already, all towers have sizes between $2$ and $2k-2$.
Since exchange is impossible, every two towers of the same size have the same structure.
But if Pusher cannot advance as well, there is at most one tower of any size.
There are $2k-2$ towers and only $2k-3$ possible sizes, which gives a contradiction.

Finally, we bound the total number of chips transferred from the buckets during the whole game.
We assign weight $F_s$ to every tower of size $s$.
Let $\TW$ denote the total weight of all towers.
Observe that every rebuild action increases $\TW$ by at least $1$.
The other two actions do not decrease $\TW$.
During exchange, the total weight increases, as $2 F_s < F_{s+i} + F_{s-i}$.
During advance, the total weight does not change, as $2 F_s = F_{s+1} + F_{s-2}$.

Let $K = (2k-2)F_{k+1} + F_{k+2} + \ldots + F_{2k-2}$.
We prove that when $\TW$ exceeds $K$, Pusher does not use rebuild actions anymore.
First observe that when $\TW> K$, then there are two towers of equal size $s \geq k+2$.
Otherwise, the total weight of towers of size at least $k+2$ would be at most $F_{k+2} + \ldots + F_{2k-2}$ and towers of smaller size cannot contribute more than $(2k-2) \cdot F_{k+1}$ to the total weight.
A rebuild action never follows after an exchange action.
Any advance action that uses towers of size at least $k+2$ creates towers of size at least $k$.
Such towers must, however, have chips on both sides, or Pusher immediately wins the game.
Hence, if $\TW$ is greater than $K$ then no more chips are transferred from the buckets.

Every chip transferred from a bucket increases $\TW$ by at least $1$.
By the above discussion, Pusher wins if each bucket initially contains $K+1-(2k-2)$ chips.
It is equivalent to starting the game with $K+1$ chips on each path, as initially in all towers there are altogether $2k-2$ chips on each path.
Therefore, Pusher has a winning strategy in a $(k,K +1)$-game.
Since $F_1 + F_2 + \ldots + F_{2k-2} = F_{2k}-1$, we get
\[
K + 1= (2k-2)F_{k+1} + F_{k+2} + \ldots + F_{2k-2} + 1 
\leq F_{2k} + 2k F_{k+1}.
\]
\end{proof}

The above two results establish the threshold value for $1$-restricted chip game up to a multiplicative factor.
We conjecture that the exact value of the threshold for $1$-restricted chip games is $t_1(k) = F_{2k}$.
Our modified strategy for Pusher that works in Maker-Maker-Braker setting is very close to meeting that threshold.
We have conducted some computer experiments that support this conjecture.

\bibliographystyle{plain}
\bibliography{paintability}
\end{document}